\documentclass{article}

\usepackage[latin1]{inputenc}
\usepackage[T1]{fontenc}

\usepackage{bbm}
\usepackage{subfigure}

\usepackage{amsmath}
\usepackage{amsfonts}
\usepackage{amssymb}
\usepackage{amsthm}
\usepackage{graphicx}
\usepackage{enumerate}

\newtheorem{theorem}{Theorem}

\newtheorem{lemma}[theorem]{Lemma}

\newtheorem{remark}[theorem]{Remark}
\numberwithin{equation}{section}

\newcommand{\PP}{\mathbb{P}}

\newcommand{\EE}{\mathbb{E}}

\begin{document}
\title{SLE(6) and the geometry of diffusion fronts}
\author{Pierre Nolin}
\date{Courant Institute of Mathematical Sciences}
\maketitle


\begin{abstract}
We study the diffusion front for a natural two-dimensional model where many particles starting at the origin diffuse independently. It turns out that this model can be described using properties of near-critical percolation, and provides a natural example where critical fractal geometries spontaneously arise.
\end{abstract}

\section{Introduction}

In this paper, we study a simple two-dimensional model on a planar lattice where a large number of particles, all starting at a given site, diffuse independently. As the particles evolve, a concentration gradient appears, and the random interfaces that arise can be described by comparing the model to an inhomogeneous percolation model (where the probability that a site is occupied or vacant depends on its location). We exhibit in particular a regime where the (properly rescaled) interfaces are fractal with dimension $7/4$: this model thus provides a natural setting where a \emph{(stochastically-)fractal geometry} spontaneously appears, as observed by statistical physicists, both on numerical simulations and in real-life situations.


To our knowledge, the study of the geometry of such ``diffusion fronts'' has been initiated by the physicists Gouyet, Rosso and Sapoval in \cite{Sa1}. In that 1985 paper, they showed numerical evidence that such interfaces are fractal, and they measured the dimension $D_f = 1.76 \pm 0.02$. To carry on simulations, they used the approximation that the status of the different sites (occupied or not) are independent of each other: they computed the probability $p(z)$ of presence of a particle at site $z$, and introduced the gradient percolation model, an inhomogeneous percolation process with occupation parameter $p(z)$. Gradient percolation provides a model for phenomena like diffusion or chemical etching -- the interface created by welding two pieces of metal for instance -- and it has since then been applied to a wide variety of situations (for more details, see the recent account \cite{Sa2} and the references therein).

Based on the observations that the front remains localized where $p(z) \simeq p_c$ and that the dimension $D_f$ was close to $7/4$, a value already observed for critical (homogeneous) percolation interfaces, they argued that there should be a close similarity between these diffusion fronts and the boundary of a large percolation cluster. For critical percolation, this conjectured dimension $7/4$ was justified later via theoretical physics arguments (see for instance \cite{DS,Ca,AAD}). In 2001, thanks to Smirnov's proof of conformal invariance \cite{Sm} and to SLE computations by Lawler, Schramm and Werner (see e.g. \cite{LSW1,LSW2}), a complete rigorous proof appeared \cite{B} in the case of the triangular lattice. It is actually known that in this case, the interfaces are described by the so-called SLE(6) process, a process which has Hausdorff dimension $7/4$. There is also a discrete counterpart of this result: the length of a discrete interface follows a power law with exponent $7/4$, that comes from the arm exponents derived in \cite{SmW}. For more details, the reader can consult the references \cite{L,W2}.



In the paper \cite{N1}, we have studied the gradient percolation model from a mathematical perspective, building on the works of Lawler, Schramm, Smirnov and Werner. Combining their results with Kesten's scaling relations \cite{Ke4}, one first gets a rather precise description of homogeneous percolation near criticality in two dimensions \cite{SmW,N2}, for instance the power law for the so-called characteristic length
$$\xi(p) = |p-1/2|^{-4/3+o(1)}$$
as $p \to 1/2$ (here and in the sequel we restrict ourselves to the triangular lattice $\mathbb{T}$). Estimates of this type (and the underlying techniques) then enabled us to verify the description in \cite{Sa1} for gradient percolation, and to confirm rigorously that dimension $7/4$ shows up in this context.


The goal of the present paper is to investigate ``diffusion front'' models mathematically, and to use our results concerning gradient percolation to show that their geometry and roughness can also be rigorously described.

Let us now present our main model in more detail. We start at time $t=0$ with a large number $n$ of particles located at the origin, and we let them perform independent (simple) random walks. At each time $t$, we then look at the sites containing at least one particle. These occupied sites can be grouped into connected components, or ``clusters'', by connecting two occupied sites if they are adjacent on the lattice.

\begin{figure}
\begin{center}

\subfigure[$t=10$.]{\includegraphics[width=.325\textwidth,clip]{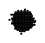}}
\subfigure[$t=100$.]{\includegraphics[width=.325\textwidth,clip]{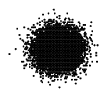}}
\subfigure[$t=500$.]{\includegraphics[width=.325\textwidth,clip]{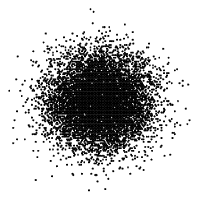}}\\

\subfigure[$t=1000$.]{\includegraphics[width=.325\textwidth,clip]{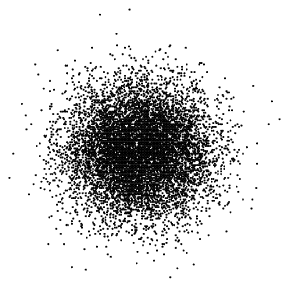}}
\subfigure[$t=1463 \:\: (=\lfloor \lambda_{\max} n \rfloor)$.]{\includegraphics[width=.325\textwidth,clip]{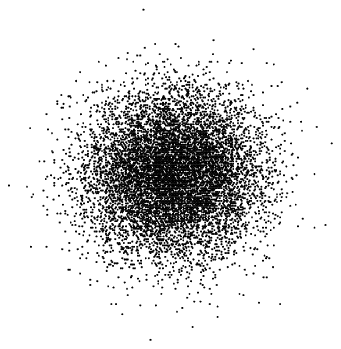}}
\subfigure[$t=2500$.]{\includegraphics[width=.325\textwidth,clip]{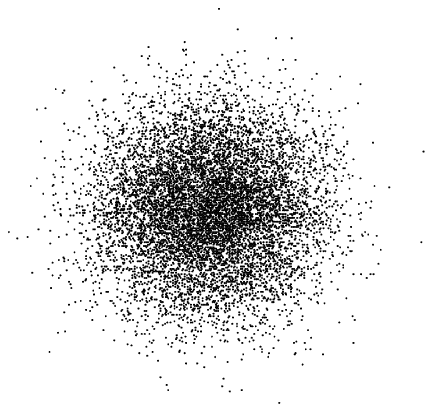}}\\

\subfigure[$t=3977 \:\: (=\lfloor \lambda_c n \rfloor)$.]{\includegraphics[width=.325\textwidth,clip]{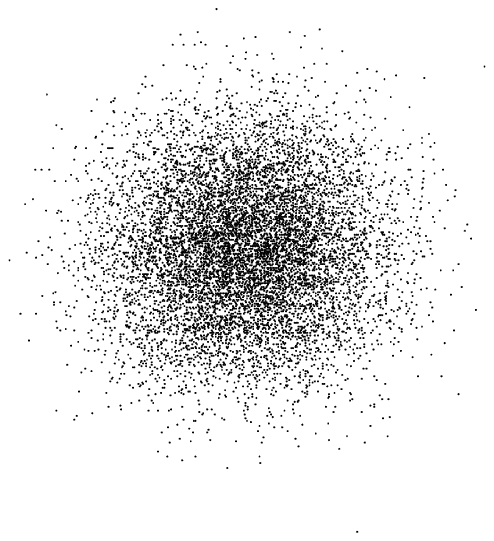}}
\subfigure[$t=5000$.]{\includegraphics[width=.325\textwidth,clip]{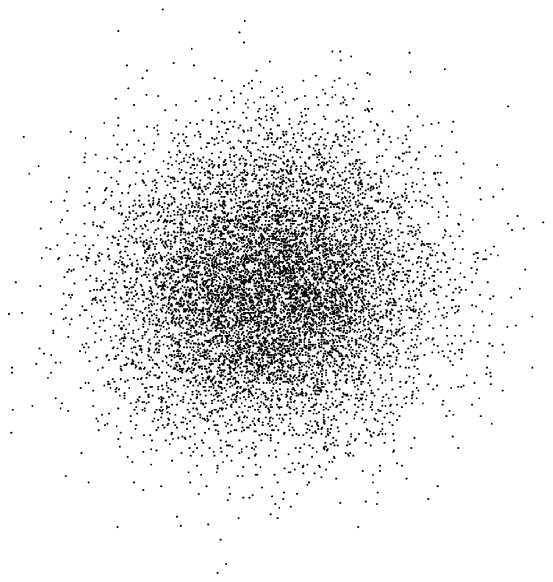}}
\subfigure[$t=10000$.]{\includegraphics[width=.325\textwidth,clip]{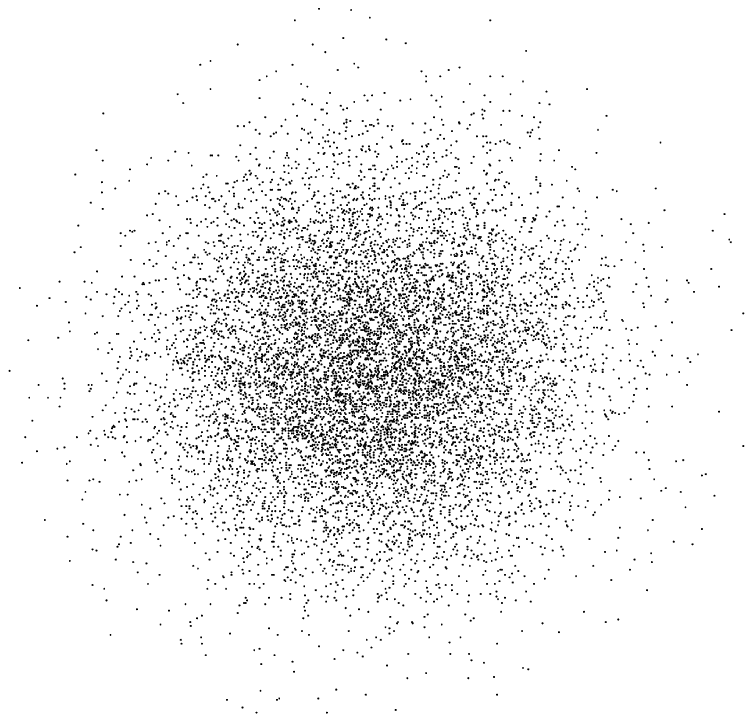}}\\

\caption{\label{evolution} Different stages of evolution as the time $t$ increases, for $n=10000$ particles.}
\end{center}
\end{figure}

As time $t$ increases, different regimes arise, as depicted in Figure \ref{evolution}. Roughly speaking and in the large $n$-limit, at first, a very dense cluster around the origin forms. This clusters grows as long as $t$ remains small compared to $n$. When $t$ gets comparable to $n$, the cluster first continues to grow up to some time $t_{\max} = \lambda_{\max} n$, then it starts to decrease, and it finally dislocates at some critical time $t_c = \lambda_c n$ -- and it never re-appears. The constants $\lambda_{\max} < \lambda_c$ can be explicitly computed, and even if they depend on the underlying lattice, their ratio is (expected to be) universal, equal to $1/e$.

To be more specific, the phase transition corresponding to $\lambda_c$ can be described as follows:
\begin{itemize}
\item When $t = \lambda n$ for $\lambda < \lambda_c$, then (with probability going to one as $n \to \infty$), the origin belongs to a macroscopic cluster. Its boundary is roughly a circle -- we will provide a formula for its radius in terms of $\lambda$ and $n$, it is of order $\sqrt{n}$ -- and its roughness can be described via the dimension $7/4$. This is a \emph{dense phase}.

\item When $t = \lambda n$ for $\lambda > \lambda_c$, then all the clusters -- in particular the cluster of the origin -- are very small, and the whole picture is stochastically dominated by a subcritical percolation picture. This is a \emph{dilute phase}.
\end{itemize}

Actually, alongside these two principal regimes, at least two other interesting cases could be distinguished: the (near-)critical regime when $t$ is very close to $\lambda_c n$, and a short-time transitory regime when $t$ is very small compared to $n$ ($t \ll n^{\alpha}$ for any $\alpha > 0$, typically $t = \log n$). We do not discuss these regimes in the present paper.

We finally study a variant of this model, where now a source creates new particles at the origin in a stationary way -- one can for instance think of the formation of an ink stain. If the new particles arrive at a fixed rate $\mu>0$, then after some time, one observes a macroscopic cluster near the origin, that grows as time passes. We explain briefly how such a model can be described using the same techniques and ideas.

To sum up, our results are thus strong evidence of \emph{universality} for random shapes in the ``real world''. They indeed indicate that as soon as a density gradient arises, random shapes similar to SLE(6) should appear -- at least under some hypotheses of regularity for the density, and of (approximate) spatial independence. However, note that it is explained in \cite{N3,NW} that these random shapes are also somewhat different from SLE(6): they are \emph{locally} asymmetric.

\section{Setup and statement of the results}

\subsection{Description of the model}

As previously mentioned, our results will be obtained by comparing our model to an (inhomogeneous) independent percolation process, for some well-chosen occupation parameters. This process is now rather well-understood, which allows one to get precise estimates on macroscopic quantities related to the interface, such as the size of its fluctuations or its length. In particular, the interface is a fractal object of dimension $7/4$, whose geometry is close to the geometry of SLE(6), the Schramm-Loewner Evolution with parameter $6$.

Recall that site percolation on a lattice $G = (V,E)$, where $V$ is the set of vertices (or ``sites''), and $E$ is the set of edges (or ``bonds'') can be described as follows. We consider a family of parameters $(p(z),z \in V)$, and we declare each site $z$ to be occupied with probability $p(z)$, vacant with probability $1 - p(z)$, independently of the other sites. We are then often interested in the connectivity properties of the set of occupied sites, that we group into connected components, or ``clusters''. In the following, we use the term ``interface'' for a curve on the dual hexagonal lattice bordered by occupied sites on one side, and by vacant sites on the other side (the boundary of a finite cluster for instance).

In the standard homogeneous case where $p(z) \equiv p \in [0,1]$, it is known for a large class of lattices that there is a phase transition at a certain critical parameter $p_c \in (0,1)$ (depending on the lattice), and it is a celebrated result of Kesten \cite{Ke0} that $p_c=1/2$ on the triangular lattice $\mathbb{T} = (\mathbb{V}^T,\mathbb{E}^T)$: when $p \leq 1/2$ there is (a.s.) no infinite cluster, while for $p > 1/2$ there is (a.s.) a unique infinite cluster. At the transition exactly, for the value $p=p_c$ of the parameter, lies the critical regime. This regime turns out to be conformally invariant in the scaling limit, a very strong property established by Smirnov \cite{Sm} that allows one to prove that cluster boundaries converge to SLE(6). One can then use SLE computations \cite{LSW1,LSW2} to derive most critical exponents associated with percolation \cite{SmW}.

From now on, we focus on the lattice $\mathbb{T}$, since at present this is the lattice for which the most precise properties, like the value of the critical exponents, are known rigorously -- a key step being the previously-mentioned conformal invariance property \cite{Sm}, a proof of which is still missing on other lattices. However, let us mention that most results remain valid on other common regular lattices like $\mathbb{Z}^2$, except those explicitly referring to the value of critical exponents, like the fractal dimension of the boundary (see Section 8.1 in \cite{N2} for a discussion of this issue). 

As usual in the statistical physics literature, the notation $f \asymp g$ means that there exist two constants $C_1,C_2 > 0$ such that $C_1 g \leq f \leq C_2 g$, while $f \approx g$ means that ${\log f} / {\log g} \to 1$ -- we also sometimes use the notation ``$\simeq$'', which has no precise mathematical meaning.

Let us now describe the model itself. We assume that at time $t=0$, a fixed number $n$ of particles are deposited at the origin, and we let these particles perform independent random walks on the lattice, with equal transition probabilities $1/6$ on every edge. We then look, at each time $t$, at the ``occupied'' sites containing at least one particle -- note that we do not make any \emph{exclusion} hypothesis, a given site can contain several particles at the same time. We denote by $\pi_t = (\pi_t(z), z \in \mathbb{V}^T)$ the distribution after $t$ steps of a simple random walk starting from $0$, so that the probability of occupation for a site $z$ is
\begin{equation}
\tilde{p}_{n,t}(z) = 1 - (1 - \pi_t(z))^n \simeq 1 - e^{- n \pi_t(z)},
\end{equation}
and we expect
$$\pi_t(z) \simeq \frac{C_T}{t} e^{- \|z\|^2/t}$$
for some constant $C_T$ depending only on the lattice. For our purpose, we will need a rather strong version of this Local Central Limit Theorem type estimate, that we state in the next section. We also denote by $N_{n,t}(z)$ the number of particles at site $z$, at time $t$.

Roughly speaking, $\tilde{p}_{n,t}(z)$ is maximal for $z=0$, and $\tilde{p}_{n,t}(0)$ decreases to $0$ as the particles evolve. Hence, different regimes arise according to $t$: if $p_{n,t}(0) > 1/2$, the picture near the origin is similar to super-critical percolation. There is thus a giant component around the origin, whose boundary is located in a region where the probability of occupation is close to $p_c=1/2$. On the other hand, if $p_{n,t}(0) < 1/2$ the model behaves like sub-critical percolation, so that all the different clusters are very small.

\subsection{Main ingredients} \label{ingredients}

\subsubsection*{Local Central Limit Theorem}

To describe the fine local structure of the boundary, we need precise estimates for $\pi_t$. We use the following form of the Local Central Limit Theorem:
\begin{equation} \label{LCLT}
\pi_t(z) = \frac{\sqrt{3}}{2 \pi t} e^{- \|z\|^2/t} \Bigg[ 1 + O\bigg( \frac{1}{t^{3/4}} \bigg) \Bigg]
\end{equation}
\emph{uniformly} for $z \in \mathbb{V}^T$, $\|z\| \leq t^{9/16}$. This equation means that locally, the estimate given by the Central Limit Theorem holds uniformly: the constant $\sqrt{3}/2$ comes from the ``density'' of sites on the triangular lattice.

This result is sharper than the estimates stated in the classic references \cite{Sp, L_RW}, in particular for atypical values like $C \sqrt{t \log t}$, and such a strong form is needed below. It can be found in Chapter 2 of \cite{LL}, where estimates are derived for random walks with increment distribution having an exponential moment. In that book, results are stated for $\mathbb{Z}^2$, but they can easily be translated into results for the triangular lattice $\mathbb{T}$ by considering a linear transformation of $\mathbb{T}$ onto $\mathbb{Z}^2$. One can take for instance the mapping $x + e^{i\pi/3} y \mapsto x + i y$ (in complex notation), which amounts to embedding $\mathbb{T}$ onto $\mathbb{Z}^2$ by adding NW-SE diagonals on every face of $\mathbb{Z}^2$ -- one then has to consider the random walk with steps uniform on $\{(\pm 1,0),(0,\pm 1),(\pm 1, \mp 1)\}$.

We also use several times the following a-priori upper bound on $\pi_t(z)$, providing exponential decay for sites $z$ at distance $\geq t^{1/2+\epsilon}$: for some universal constant $C>0$, we have
\begin{equation} \label{Hoeffding}
\pi_t(z) \leq C e^{-\|z\|^2/2t}
\end{equation}
for all sites $z$, for any time $t \geq 0$. This bound is for instance a direct consequence of Hoeffding's inequality.

\subsubsection*{Poisson approximation}

Note that if we take a Poissonian number of particles $N \sim \mathcal{P}(n)$, instead of a fixed number, the \emph{thinning} property of the Poisson distribution immediately implies that the random variables $(N_{n,t}(z), z \in \mathbb{V}^T)$ are independent of each other, with $N_{n,t}(z) \sim \mathcal{P}(n \pi_t(z))$. We thus simply get an independent percolation process with parameters
\begin{equation}
p_{n,t}(z) = 1 - e^{-n \pi_t(z)}.
\end{equation}

As will soon become clear, this Poisson approximation turns out to be valid, since we only need to observe a negligible fraction of the particles to fully describe the boundary: what we observe locally looks very similar to the ``Poissonian case''. For that, we use the classic Chen-Stein bound (see Theorem 2.M on p.34 in \cite{BHJ}, or \cite{CDM}). Assume that $W = X_1 + \ldots + X_k$ is a sum of $k$ independent random variables, with $X_i \sim \mathcal{B}(p_i)$. Then for $\lambda = \EE[W] = \sum_{i=1}^k p_i$, we have for the total variation norm
\begin{equation}
\big\| \mathcal{L}( W ) - \mathcal{P}(\lambda) \big\|_{TV} \leq \frac{1 - e^{-\lambda}}{\lambda} \sum_{i=1}^k p_i^2.
\end{equation}

Now, assume that we look at what happens in a (small) subset $A$ of the sites. The distribution in $A$ can be obtained by first drawing how many particles fall in $A$, and then choosing accordingly the configuration in $A$. Each particle has a probability $\pi_t(A) = \sum_{z \in A} \pi_t(z)$ to be in $A$, so that the number of particles in $A$ is $N_t(A) \sim \mathcal{B}(n,\pi_t(A))$, and the Chen-Stein bound implies that ($\lambda = n \pi_t(A)$ in this case)
\begin{equation}
\big\| \mathcal{L}( N_t(A) ) - \mathcal{P}( n \pi_t(A) ) \big\|_{TV} \leq \frac{1}{n \pi_t(A)} n \pi_t(A)^2 = \pi_t(A).
\end{equation}
We can thus couple $N_t(A)$ with $N_A \sim \mathcal{P}( n \pi_t(A) )$ so that they are equal with probability at least $1 - \pi_t(A)$, which tends to $1$ if $A$ is indeed negligible. Hence to summarize, we can ensure that the configuration in $A$ coincides with the ``Poissonian configuration'', i.e. with the result of an independent percolation process, with probability at least $1 - \pi_t(A)$.

\subsubsection*{Radial gradient percolation}

We use earlier results on near-critical percolation, especially Kesten's paper \cite{Ke4}. All the results that we will use here are stated and derived in \cite{N2}, and we follow the notations of that paper. In particular, we use the basis $(1,e^{i\pi/3})$, so that $[a_1,a_2]\times[b_1,b_2]$ refers to the parallelogram with corners $a_i+b_j e^{i\pi/3}$ ($i,j \in \{1,2\}$).

We recall that the characteristic length is defined as
\begin{equation}
L_{\epsilon}(p) = \min\{n \text{\: s.t. \:} \PP_p(\mathcal{C}_H([0,n] \times [0,n])) \leq \epsilon \}
\end{equation}
when $p<1/2$ ($\mathcal{C}_H$ denoting the existence of a horizontal occupied crossing), and the same with vacant crossings when $p>1/2$ (so that $L_{\epsilon}(p) = L_{\epsilon}(1-p)$), for any $\epsilon \in (0,1/2)$, and that we have
\begin{equation}
L_{\epsilon}(p) \asymp L_{\epsilon'}(p) \approx |p-1/2|^{-4/3}
\end{equation}
for any two fixed $0 < \epsilon, \epsilon' < 1/2$. We will also use the property of exponential decay with respect to $L$: for any $\epsilon \in (0,1/2)$ and any $k \geq 1$, there exist some universal constants $C_i = C_i(k,\epsilon)$ such that
\begin{equation} \label{expdecay_cross}
\PP_p(\mathcal{C}_H([0,n] \times [0,k n])) \leq C_1 e^{-C_2 n / L_{\epsilon}(p)}
\end{equation}
uniformly for $p < 1/2$ and $n \geq 1$. \emph{In the following, we fix a value of $\epsilon$, for instance $\epsilon=1/4$, and we simply write $L$, forgetting about the dependence on $\epsilon$.}

The aforementioned approximation allows to use for our model results concerning the independent Bernoulli case: we hence consider a ``radial'' gradient percolation, a percolation process where the parameter decreases with the distance to the origin. Let us first recall the description obtained in a strip $[0,\ell_N] \times [0,N]$ with a constant vertical gradient, i.e. with parameter $p(y) = 1 - y/N$: we will explain later how to adapt these results in our setting. The characteristic phenomenon is the appearance of a unique ``front'', an interface touching at the same time the occupied cluster at the bottom of the strip and the vacant cluster at the top: if $\ell_N \geq N^{4/7+\delta}$ for some $\delta>0$, this happens with high probability (tending to $1$ as $n \to \infty$). Moreover, for any fixed small $\epsilon>0$, also with high probability,
\begin{enumerate}[(i)]
\item this front remains localized around the line $y = y_c = \lfloor N/2 \rfloor$, where $p(y) \simeq 1/2$, and its fluctuations compared to this line are of order
\begin{equation}
\sigma_N = \sup \big\{ \sigma \text{ s.t. } L(p(y_c + \sigma)) \geq \sigma \big\} \approx N^{4/7},
\end{equation}
more precisely they are smaller than $N^{4/7+\epsilon}$ (the front remains in a strip of width $2N^{4/7+\epsilon}$ around $y=y_c$), and larger than $N^{4/7-\epsilon}$ (both above and below $y=y_c$),

\item and since it remains around the line $y=y_c$, its behavior is comparable to that of critical percolation, so that in particular its fine structure can be described via the exponents for critical percolation: for instance, its discrete length $L_N$ satisfies
\begin{equation}
N^{3/7-\epsilon} \ell_N \leq L_N \leq N^{3/7+\epsilon} \ell_N.
\end{equation}

\end{enumerate}

\subsection{Statement of the results}

In view of the previous sections, let us introduce
\begin{equation}
\bar{\pi}_t(r) = \frac{\sqrt{3}}{2 \pi t} e^{- r^2/t}
\end{equation}
(so that $\pi_t(z) = \bar{\pi}_t(\|z\|) [1+ O(1/t^{3/4})]$ uniformly for $\|z\| \leq t^{9/16}$ by Eq.(\ref{LCLT})), and the associated ``Poissonian'' parameters
\begin{equation}
\bar{p}_{n,t}(r) = 1 - e^{-n \bar{\pi}_t(r)}.
\end{equation}

Since $\bar{p}_{n,t}$ is a decreasing continuous function of $r$, and tends to $0$ when $r \to \infty$, there is a unique $r$ such that $\bar{p}_{n,t}(r) = 1/2$ \emph{iff} $\bar{p}_{n,t}(0) \geq 1/2$, or equivalently \emph{iff} $t \leq t_c = \lambda_c n$, with
\begin{equation}
\lambda_c = \frac{\sqrt{3}}{2 \pi \log 2} \quad (\simeq 0.397\ldots).
\end{equation}
We introduce in this case
\begin{equation}
r_{n,t}^* = (\bar{p}_{n,t})^{-1}(1/2),
\end{equation}
where $(\bar{p}_{n,t})^{-1}$ denotes the inverse function of $\bar{p}_{n,t}$. In particular, $r_{n,t}^* = 0$ when $t = \lambda_c n$. One can easily check that 
\begin{equation} \label{rstar}
r_{n,t}^* = \sqrt{t \log \frac{\lambda_c}{t/n}},
\end{equation}
and a small calculation shows that $r_{n,t}^*$, as a function of $t$, increases on $(0,\lambda_{\max} n]$ and then decreases on $[\lambda_{\max} n, \lambda_c n]$ (finally being equal to $0$), with
\begin{equation}
\lambda_{\max} = \frac{\lambda_c}{e} = \frac{\sqrt{3}}{2 e \pi \log 2} \quad (\simeq 0.146\ldots).
\end{equation}
In particular, the ratio $t_{\max} / t_c$ is equal to $1/e$, and is thus independent of $n$ -- and though the lattice appears in $\lambda_c$ and $\lambda_{\max}$, as the ``density'' of sites, $\sqrt{3}/2$ here, it disappears when we take the ratio $t_{\max} / t_c$. For future reference, note also that
\begin{itemize}
\item when $t = \lambda n$, for some fixed $\lambda < \lambda_c$, then $r_{n,t}^*$ increases as $\sqrt{t} \asymp \sqrt{n}$:
$$r_{n,t}^* = \sqrt{\log \frac{\lambda_c}{\lambda}} \sqrt{t},$$

\item when $t = n^{\alpha}$, for some $0 < \alpha < 1$,
$$r_{n,t}^* = \sqrt{t} \sqrt{\bigg( \frac{1}{\alpha} - 1 \bigg) \log t + \log \lambda_c} = \sqrt{\frac{1}{\alpha} - 1} \sqrt{t \log t} + O \bigg( \frac{\sqrt{t}}{\sqrt{\log t}}\bigg).$$
\end{itemize}

We also introduce
\begin{equation}
\sigma_{n,t}^{\pm} = \sup \big\{ \sigma \text{ s.t. } L(p(\lfloor r_{n,t}^* \rfloor \pm \sigma)) \geq \sigma \big\}
\end{equation}
that will measure the fluctuations of the interface. We will see that $\sigma_{n,t}^{\pm} \approx t^{2/7}$ under suitable hypotheses on $n$ and $t$. We are now in a position to state our main result. We denote by $S_r$ (resp. $S_{r,r'}$) the circle of radius $r$ (resp. the annulus of radii $r < r'$) centered at the origin.

\begin{theorem} \label{main_theorem}
Consider some sequence $t_n \to \infty$. Then
\begin{enumerate}[(i)]
\item If $t_n \geq \lambda n$ for some $\lambda > \lambda_c$, then there exists a constant $c = c(\lambda)$ such that
$$\mathbb{P}_{n,t}(\text{every cluster is of diameter $\leq c \log n$}) \to 1$$
as $n \to \infty$. \label{dilute}

\item If $e^{(\log n)^{\alpha}} \leq t_n \leq \lambda' n$ for some $\alpha \in (0,1)$ and $\lambda' < \lambda_c$, then for any fixed small $\epsilon > 0$, the following properties hold with probability tending to $1$ as $n \to \infty$: \label{dense}
\begin{enumerate}
\item there is a unique interface $\gamma_n$ surrounding $\partial S_{t_n^{\epsilon}}$,

\item this interface remains localized in the annulus of width $2t_n^{2/7+\epsilon}$ around $\partial S_{r_{n,t_n}^*}$ (recall that $r_{n,t_n}^* \approx \sqrt{t_n}$), and compared to this circle, its fluctuations, both inward and outward, are larger than $t_n^{2/7-\epsilon}$,

\item its length -- number of edges -- $L_n$ satisfies $t_n^{5/7-\epsilon} \leq L_n \leq t_n^{5/7+\epsilon}$ (in other words, it behaves as $L_n \approx t_n^{5/7}$).
\end{enumerate}

\end{enumerate}
\end{theorem}

\section{Proofs}

From now on, we choose not to mention the dependence of $t$ on $n$ in most instances, for notational convenience.

\subsubsection*{Case (\ref{dilute}): $t_n \geq \lambda n$ for some $\lambda > \lambda_c$}

Let us start with case (\ref{dilute}), which is the simpler one. In this case, the hypothesis $t_n \geq \lambda n$ implies that for any site $z$,
\begin{equation} \label{bound_subcrit}
\bar{p}_{n,t}(\|z\|) \leq \bar{p}_{n,t}(0) \leq 1 - e^{- \frac{\lambda_c}{\lambda} \log 2} = \rho < 1/2.
\end{equation}
For $z$ such that $\|z\| \leq t^{9/16}$, since
$$\pi_t(z) = \bar{\pi}_t(\|z\|) [1+ O(1/t^{3/4})] = \bar{\pi}_t(\|z\|) + O(1/n^{7/4})$$
uniformly, we also have
\begin{equation} \label{domination}
p_{n,t}(z) = 1 - e^{-n \pi_t(z)} = 1 - e^{-n \bar{\pi}_t(\|z\|) + O(1/n^{3/4})} \leq \rho' < 1/2
\end{equation}
for some $\rho' > \rho$. Moreover, this bound also holds for $\|z\| \geq t^{9/16}$ by Eq.(\ref{Hoeffding}).

This enables us to dominate our model by a subcritical percolation process. Take some $\delta>0$, and consider the model with $N \sim \mathcal{P}((1+\delta) n)$ particles: $N \geq n$ with probability tending to $1$, and in this case the resulting configuration dominates the original one. Moreover, if we choose $\delta$ such that
$$\lambda_c (1+\delta) n \leq \frac{\lambda+\lambda_c}{2} n,$$
the hypothesis $t_n \geq \tilde{\lambda} (1+\delta) n$ holds for some $\tilde{\lambda} \in (\lambda_c,\lambda)$, so that the bound (\ref{domination}) is satisfied by the associated parameters for some $\rho'' \in (\rho',1/2)$.

Let us assume furthermore that $t \leq n^3$, and subdivide $S_t$ into $\sim t^2$ horizontal and vertical parallelograms of size $\frac{c}{2} \log n \times c \log n$. Any cluster of diameter larger than $c \log n$ would have to cross one of these parallelograms ``in the easy direction'', which -- using the exponential decay property of percolation -- has a probability at most
\begin{equation}
C_0 t^2 e^{- C(\rho'') c \log n} \leq C_0 n^{6 - C(\rho'') c},
\end{equation}
that tends to $0$ as $n \to \infty$ for $c$ large enough.

In fact, the case $t \geq n^3$ can be handled more directly. We dominate our model by using $N \sim \mathcal{P}(2n)$: $N \geq n$ with high probability, and we have
$$\pi_t(z) = \bar{\pi}_t(\|z\|) [1+ O(1/t^{3/4})] = O(1/n^3),$$
so that $p_{2n,t}(z) = O(1/n^2)$. Hence, the probability to observe two neighboring sites is at most
$$\sum_{z \in \mathbb{V}^T} \sum_{z' \sim z} p_{2n,t}(z) p_{2n,t}(z') \leq 6 \max_{z' \in \mathbb{V}^T} p_{2n,t}(z') \times \Bigg( \sum_{z \in \mathbb{V}^T} p_{2n,t}(z) \Bigg) = O\bigg(\frac{1}{n}\bigg),$$
since $\sum_{z \in \mathbb{V}^T} p_{2n,t}(z) \leq 2n$.

\subsubsection*{Case (\ref{dense}): $e^{(\log n)^{\alpha}} \leq t_n \leq \lambda' n$ for some $\alpha \in (0,1)$ and $\lambda' < \lambda_c$}

Let us turn now to case (\ref{dense}). First of all, the hypothesis on $t_n$ ensures that
\begin{equation}
C_{\lambda'} \sqrt{t} \leq r_{n,t}^* \leq C_{\alpha} (\log t)^{1/2\alpha} \sqrt{t}
\end{equation}
Let us first study the behavior of $\bar{p}_{n,t}$ around $r_{n,t}^*$: we have
\begin{equation} \label{derivative}
\frac{\partial}{\partial r} \bar{p}_{n,t}(r) = (1 - \bar{p}_{n,t}(r)) \times \big( \log (1 - \bar{p}_{n,t}(r)) \big) \times \bigg( \frac{2 r}{t} \bigg).
\end{equation}
The hypothesis $t_n \leq \lambda' n$ also implies that
\begin{equation}
\bar{p}_{n,t}(0) \geq 1/2 + 2 \delta
\end{equation}
for some $\delta >0$ (like in Eq.(\ref{bound_subcrit})). We can thus define
\begin{equation}
r_{n,t}^- = (\bar{p}_{n,t})^{-1}(1/2+\delta) \quad \text{ and } \quad r_{n,t}^+ = (\bar{p}_{n,t})^{-1}(1/2-\delta)
\end{equation}
as we defined $r_{n,t}^*$, so that $\bar{p}_{n,t}(r)$ is bounded away from $0$ and $1$ for $r \in [r_{n,t}^-,r_{n,t}^+]$. Note that it also implies that $n \bar{\pi}_t(r)$ is bounded away from $0$ and $\infty$. A formula similar to Eq.(\ref{rstar}) apply in this case, in particular we have
\begin{equation}
C_{\lambda'}^{\pm} \sqrt{t} \leq r_{n,t}^{\pm} \leq C_{\alpha}^{\pm} (\log t)^{1/2\alpha} \sqrt{t}.
\end{equation}
Hence, we get from Eq.(\ref{derivative}) that for $r \in [r_{n,t}^-,r_{n,t}^+]$,
\begin{equation}
- \frac{C_1 (\log t)^{1/2\alpha}}{\sqrt{t}} \leq \frac{\partial}{\partial r} \bar{p}_{n,t}(r) \leq - \frac{C_2}{\sqrt{t}}.
\end{equation}
We deduce
\begin{equation} \label{ineq_p}
\frac{1}{2} - \frac{C_1 (\log t)^{1/2\alpha}}{\sqrt{t}} \big(r - r_{n,t}^*\big) \leq \bar{p}_{n,t}(r) \leq \frac{1}{2} - \frac{C_2}{\sqrt{t}} \big(r - r_{n,t}^*\big)
\end{equation}
for $r \in [r_{n,t}^*,r_{n,t}^+]$, and the same is true with inequalities reversed for $r \in [r_{n,t}^-,r_{n,t}^*]$. Now, for any $z$ in the annulus $S_{r_{n,t}^-,r_{n,t}^+}$,
\begin{equation}
p_{n,t}(z) - \bar{p}_{n,t}(z) = e^{-n \bar{\pi}_t(\|z\|)} \big(1 - e^{n(\bar{\pi}_t(\|z\|) - \pi_t(z))}\big),
\end{equation}
and using Eq.(\ref{LCLT}), we get $n (\bar{\pi}_t(\|z\|) - \pi_t(z)) = (n \bar{\pi}_t(\|z\|)) \times O(1/t^{3/4}) = O(1/t^{3/4})$, so that
\begin{equation} \label{bound_param}
p_{n,t}(z) - \bar{p}_{n,t}(z) = O(1/t^{3/4}).
\end{equation}
Hence, Eq.(\ref{ineq_p}) is also valid with $p_{n,t}(z)$, taking different constants if necessary:
\begin{equation}
p_{n,t}(z) = \frac{1}{2} - \frac{\epsilon(z)}{\sqrt{t}} \big(\|z\| - r_{n,t}^*\big),
\end{equation}
where $C'_1 \leq \epsilon(z) \leq C'_2 (\log t)^{1/2\alpha}$.

\begin{figure}
\begin{center}
\includegraphics[width=12cm]{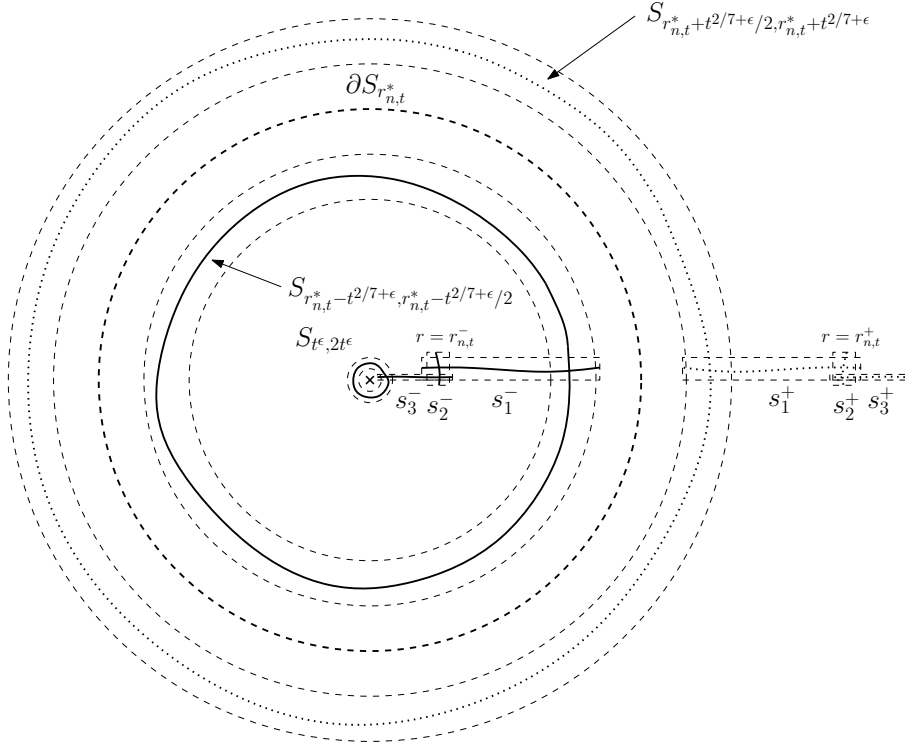}
\caption{\label{gradient} Construction showing localization of the macroscopic interface in $S_{r_{n,t}^* - t^{2/7+\epsilon},r_{n,t}^* + t^{2/7+\epsilon}}$.}
\end{center}
\end{figure}

We now apply a Poissonian approximation, as explained in Section \ref{ingredients}, in a well-chosen subset $A$ of sites -- typically an annulus of width $t^{2/7+\epsilon}$ around $\partial S_{r_{n,t}^*}$, but we need to enhance it a little bit. More precisely, we consider the configuration depicted on Figure \ref{gradient}: we take as a subset $A$ of sites the annulus $S_{r_{n,t}^* - 2 t^{2/7+\epsilon},r_{n,t}^* + 2 t^{2/7+\epsilon}}$, together with the smaller annulus $S_{t^{\epsilon},2t^{\epsilon}}$ and some strips connecting these annuli, $s_1^-=[r_{n,t}^-/2,r_{n,t}^* - 0.99 \times t^{2/7+\epsilon}] \times [0,t^{2/7+\epsilon}]$, $s_3^-=[0,r_{n,t}^-] \times [0,t^{\epsilon}]$ and some vertical strip $s_2^-$ in the middle. We also add external strips, $s_1^+=[r_{n,t}^* + 0.99 \times t^{2/7+\epsilon},2r_{n,t}^+] \times [0,t^{2/7+\epsilon}]$ and $s_3^+=[r_{n,t}^+,t] \times [0,t^{\epsilon}]$, connected by a vertical strip $s_2^+$.

We claim that for this choice of $A$, we have
\begin{equation}
\pi_t(A) = O(t^{-3/14 + 2 \epsilon}).
\end{equation}
Indeed, this is a direct consequence of $\pi_t(z) = O(1/t)$ and $r_{n,t}^* = O(\sqrt{t \log t})$, except for the most external strip $s_3^+=[r_{n,t}^+,t] \times [0,t^{\epsilon}]$: for this strip, we have (recall that $r_{n,t}^+ \geq C_{\lambda'}^+ \sqrt{t}$)
\begin{align*}
\pi_t(s_3^+) & = O \Bigg( t^{\epsilon} \sum_{r = C_{\lambda'}^+ \sqrt{t}}^{t^{9/16}} \frac{1}{t} e^{-r^2/t} \Bigg) + O \Big( t^{\epsilon} \times t \times e^{-t^{1/8}/2} \Big)\\
& = O ( t^{-1/2+\epsilon} ),
\end{align*}
using Eqs.(\ref{LCLT}) and (\ref{Hoeffding}).

In both annuli $S_{r_{n,t}^* - 2 t^{2/7+\epsilon},r_{n,t}^* - t^{2/7+\epsilon}}$ and $S_{r_{n,t}^* + t^{2/7+\epsilon},r_{n,t}^* + 2 t^{2/7+\epsilon}}$, the correlation length is at most
\begin{equation}
L\bigg(\frac{1}{2} \pm \frac{C_2}{\sqrt{t}} t^{2/7+\epsilon} \bigg) = L\bigg(\frac{1}{2} \pm C_2 t^{-3/14+\epsilon} \bigg) \approx t^{2/7 - 4\epsilon/3},
\end{equation}
the property of exponential decay with respect to $L$ (Eq.(\ref{expdecay_cross})) thus implies that there are circuits in these annuli, an occupied circuit in the internal one and a vacant circuit in the external one: indeed, these circuits can be constructed with roughly $r_{n,t}^*/t^{2/7} \approx t^{3/14}$ overlapping parallelograms of size $2 t^{2/7} \times t^{2/7}$ (this size is at the same time $\gg t^{2/7 - 4\epsilon/3}$, so that crossing probabilities tend to $1$ sub-exponentially fast, and $\ll t^{2/7+\epsilon}$). For the same reason, there are also occupied crossings in $s_1^-$ and $s_2^-$, and vacant crossings in $s_1^+$ and $s_2^+$. Since the parameter $p_{n,t}$ is at least $1/2 + \delta$ in $S_{r_{n,t}^-}$, there is also an occupied crossing in $s_3^-$, and an occupied circuit in $S_{t^{\epsilon},2t^{\epsilon}}$. Since $p_{n,t}$ is at most $1/2 - \delta$ outside of $S_{r_{n,t}^+}$, there is also a vacant crossing in $s_3^+$. As a consequence, the potential interfaces have to be located in the annulus $S_{r_{n,t}^* - 2 t^{2/7+\epsilon},r_{n,t}^* + 2 t^{2/7+\epsilon}}$.

Let us now consider of order $r_{n,t}^* / t^{2/7+\epsilon} \approx t^{3/14-\epsilon}$ sub-sections of length $t^{2/7+\epsilon}$ in the main annulus as on Figure \ref{uniq_figure}. If we use the same construction as for uniqueness in Proposition 7 of \cite{N1} (see also \cite{N3}), we can then show that with probability tending to $1$, there is an edge $e$ (on the dual lattice) in $S_{r_{n,t}^* - t^{2/7+\epsilon},r_{n,t}^* + t^{2/7+\epsilon}}$ with two arms, one occupied going to $\partial S_{r_{n,t}^* - 2 t^{2/7+\epsilon}}$ and one vacant to $\partial S_{r_{n,t}^* + 2 t^{2/7+\epsilon}}$. In this case, the interface is unique, and can be characterized as the set of edges possessing two such arms.

To see that the fluctuations of the interface are larger than $t^{2/7-\epsilon}$, we can use ``blocking vertical crossings'' as for Theorem 6 in \cite{N1}: the definition of $L$ indeed implies that in each rhombus of size $2 t^{2/7-\epsilon}$ centered on $\partial S_{r_{n,t}^*}$, there is a ``vertical'' crossing with probability bounded below by $1/4$.

\begin{figure}
\begin{center}
\includegraphics[width=10cm]{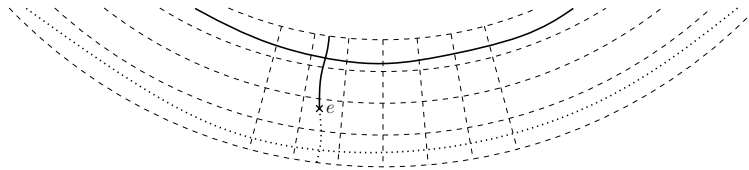}
\caption{\label{uniq_figure} By considering sub-sections of the main annulus, we can show the existence of an edge in $S_{r_{n,t}^* - t^{2/7+\epsilon},r_{n,t}^* + t^{2/7+\epsilon}}$ possessing two arms as depicted, which implies uniqueness of the interface.}
\end{center}
\end{figure}

We have just seen that the interface can be characterized as the set of two-arm edges. This allows to describe its fine local structure, by comparing it to critical percolation interfaces like in \cite{N1}. Hence, its length -- number of edges -- behaves in expectation as
\begin{equation}
\mathbb{E}[L_n] \approx t^{2/7} \sqrt{t} (t^{2/7})^{-1/4} = t^{5/7},
\end{equation}
$\alpha_2=1/4$ being the so-called ``two-arm exponent'' measuring the probability for critical percolation to observe two arms: the probability to observe two such arms going to distance $n$ decreases as $n^{-\alpha_2}$. Moreover, a decorrelation property between the different sub-sections also holds, as in \cite{N1}, so that $L_n$ is concentrated around its expectation: more precisely, $L_n / \EE[L_n] \to 1$ in $L^2$ as $n \to \infty$.

\begin{remark}
We could also construct scaling limits as in \cite{N3}, using the technology of Aizenman and Burchard \cite{AB} to prove tighness. One can for example look at the portion of the interface close to the $x$-axis and scale it by a quantity of order $t^{2/7}$ (the right quantity is called ``characteristic length for the gradient percolation model'' in \cite{N3}). The curves we get in this way are of Hausdorff dimension $7/4$, as for critical percolation.
\end{remark}

\begin{remark}
Here case (\ref{dense}) contains simultaneously the two cases $t \sim \lambda n$ (for $\lambda < \lambda_c$) and $t \sim n^{\alpha}$. However, note that these two regimes can be distinguished by the size of the ``transition window'', where the parameter decreases from $1/2+\delta$ to $1/2-\delta$, as shown on Figure \ref{densecluster}:
\begin{itemize}
\item When $t \sim \lambda n$, the transition takes place gradually with respect to $\sqrt{t}$. The parameter at the origin is bounded away from $1$, so that there are ``holes'' everywhere, even near the origin.
\item When $t \sim n^{\alpha}$, the transition is concentrated in a window of size $\sim \sqrt{t}/\sqrt{\log t}$ around $r^* \asymp \sqrt{t \log t}$. The parameter at the origin tends to $1$, and the macroscopic cluster is very dense around the origin.
\end{itemize}
Here we used that around $r_{n,t}^*$, the local behavior of $p_{n,t}$ is similar for all $t$ such that $e^{(\log n)^{\alpha}} \leq t \leq \lambda' n$, with at most an extra $(\log t)^{1/2\alpha}$ factor that does not affect the exponents.
\end{remark}

\begin{figure}
\begin{center}

\subfigure[$t \ll n$: localized transition.]{\includegraphics[width=.5\textwidth,clip]{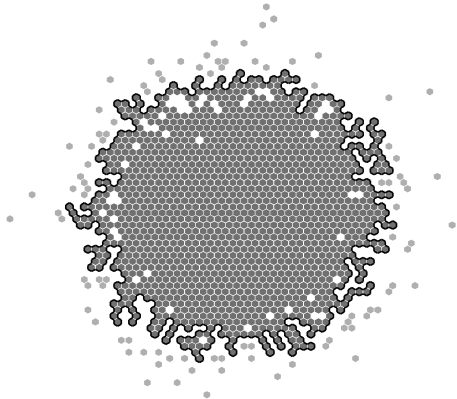}}\\

\subfigure[$t \sim n$: gradual transition.]{\includegraphics[width=.72\textwidth,clip]{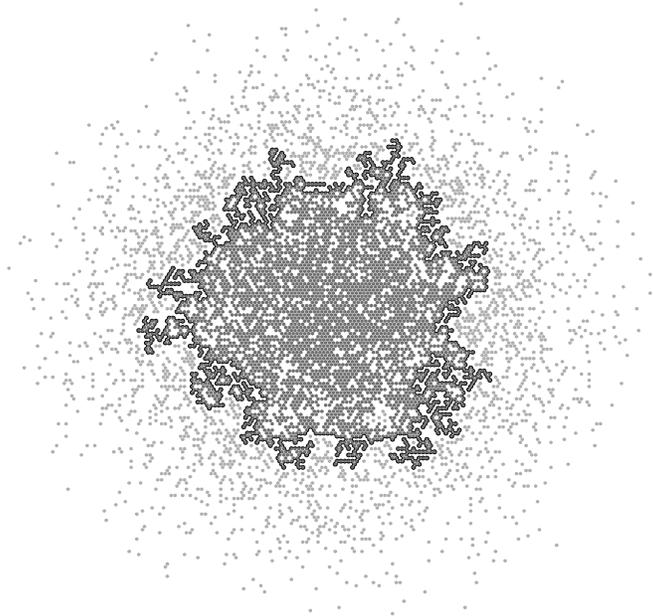}}\\

\caption{\label{densecluster} Cluster of the origin in dense phase.}
\end{center}
\end{figure}

\section{Model with a source}

We briefly discuss in this section a dynamical model with a source of particles at the origin (think for instance of an ink stain spreading out). We start with no particles, and at each time $t \geq 0$, new particles are added at the origin with rate $\mu$: the number $n_t$ of new particles at time $t$ has distribution $\mathcal{P}(\mu)$ for some fixed $\mu>0$. Once arrived, the particles perform independent simple random walks. As shown on Figure \ref{evolution2}, we observe in this setting a macroscopic cluster around the origin, whose size increases as the number of particles gets larger and larger.

\begin{figure}
\begin{center}

\subfigure[$t=10$.]{\includegraphics[width=.325\textwidth,clip]{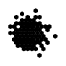}}
\subfigure[$t=100$.]{\includegraphics[width=.325\textwidth,clip]{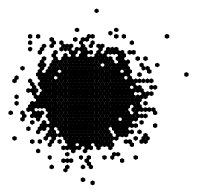}}
\subfigure[$t=1000$.]{\includegraphics[width=.325\textwidth,clip]{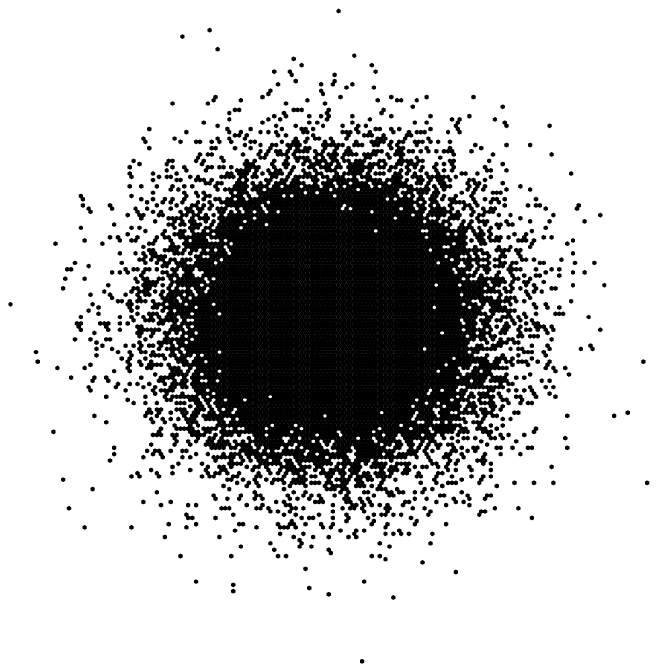}}\\

\caption{\label{evolution2} A ``stain'' growing as the time $t$ increases, with new particles arriving at rate $\mu = 50$.}
\end{center}
\end{figure}

The number of particles $N_t(z)$ on site $z$ at time $t$ is given by a sum of independent Poisson random variables:
\begin{equation}
N_t(z) \sim \mathcal{P}(\mu \pi_t(z)) + \ldots + \mathcal{P}(\mu \pi_0(z)) \sim \mathcal{P}(\mu \rho_t(z)),
\end{equation}
with
\begin{equation}
\rho_t(z) = \sum_{u=0}^t \pi_u(z).
\end{equation}
We thus introduce as for the previous model
\begin{equation}
q_{\mu,t}(z) = 1 - e^{-\mu \rho_t(z)}.
\end{equation}

For any value of $\mu$, this parameter tends to $1$ near the origin, and this dynamical model turns out to behave like the previous model in dense phase. We show that the constructions of the previous section can easily be adapted to this new setting.

\begin{lemma}
For each $\epsilon >0$, there exist constants $C_i=C_i(\epsilon)$ such that
\begin{equation}
\rho_t(z) \geq C_1 \log t - C_2
\end{equation}
uniformly on sites $z$ such that $\|z\| \leq t^{1/2-\epsilon}$.
\end{lemma}
Hence in particular, $q_{\mu,t}(z) \to 1$ uniformly on $S_{t^{1/2-\epsilon}}$.

\begin{proof}
We write
$$\rho_t(z) \geq \sum_{u=t^{1-2\epsilon}}^t \pi_u(z) \geq \sum_{u=t^{1-2\epsilon}}^t \Bigg( \frac{\sqrt{3}}{2 \pi u} e^{- \|z\|^2/u} + O \bigg( \frac{1}{u^{7/4}} \bigg) \Bigg).$$
Since $\|z\|^2/u \leq t^{1-2\epsilon}/u \leq 1$, and $\sum_{u=t^{1-2\epsilon}}^t 1/u^{7/4} \leq \sum_{u=1}^{\infty} 1/u^{7/4} < \infty$, we get
\begin{equation}
\rho_t(z) \geq \frac{\sqrt{3} e^{-1}}{2 \pi} \bigg( \sum_{u=t^{1-2\epsilon}}^t \frac{1}{u} \bigg) - C,
\end{equation}
whence the result, since
\begin{align*}
\sum_{u=t^{1-2\epsilon}}^t \frac{1}{u} & = \big(\log t + \gamma + o(1)\big) - \big((1-2\epsilon)\log t + \gamma + o(1)\big)\\
& = 2 \epsilon \log t + o(1).
\end{align*}
\end{proof}

Let us now consider for $r>0$
\begin{equation}
\bar{\rho}_t(r) = \frac{\sqrt{3}}{2 \pi} \int_{r^2/t}^{+\infty} \frac{e^{-u}}{u} du,
\end{equation}
and also
\begin{equation}
\quad \bar{q}_{\mu,t}(r) = 1 - e^{-\mu \bar{\rho}_t(r)}.
\end{equation}

\begin{lemma}
We have
\begin{equation}
\rho_t(z) = \bar{\rho}_t(\|z\|) + O\bigg( \frac{1}{t^{9/16}} \bigg)
\end{equation}
uniformly on sites $z$ such that $\|z\| \geq t^{7/16}$.
\end{lemma}

\begin{proof}
We write
\begin{align*}
\rho_t(z) & = \sum_{u=0}^{t^{3/4}} \pi_u(z) + \sum_{u=t^{3/4}+1}^t \pi_u(z)\\
& = \sum_{u=0}^{t^{3/4}} \pi_u(z) + \sum_{u=t^{3/4}+1}^t \Bigg( \frac{\sqrt{3}}{2 \pi u} e^{- \|z\|^2/u} + O \bigg( \frac{1}{u^{7/4}} \bigg) \Bigg).
\end{align*}
For each $0 \leq u \leq t^{3/4}$, we have by Eq.(\ref{Hoeffding})
\begin{equation}
\pi_u(z) \leq C e^{-\|z\|^2/2u} \leq C e^{-t^{1/8}/2}
\end{equation}
since $\|z\|^2 \geq t^{7/8}$, so that the corresponding sum is a $O(t^{3/4} e^{-t^{1/8}/2}) = o(1/t^{3/4})$. We also have
$$\sum_{u=t^{3/4}+1}^t O \bigg( \frac{1}{u^{7/4}} \bigg) = O \bigg( \frac{1}{t^{9/16}} \bigg).$$
Let us now estimate the main sum. We can write
\begin{equation}
\sum_{u=t^{3/4}+1}^t \frac{1}{u} e^{- \|z\|^2/u} = \sum_{u=t^{3/4}+1}^t \frac{1}{t} f_{\lambda}\bigg(\frac{u}{t}\bigg),
\end{equation}
with $f_{\lambda}(x) = e^{-\lambda/x} / x$ and $\lambda = \|z\|^2/t \geq t^{-1/8}$, and we have
\begin{equation}
\Bigg| \sum_{u=t^{3/4}+1}^t \frac{1}{t} f_{\lambda}\bigg(\frac{u}{t}\bigg) - \int_{t^{-1/4}}^1 f_{\lambda}(x) dx \Bigg| \leq \sum_{u=t^{3/4}+1}^t \int_{(u-1)/t}^{u/t} \Bigg| f_{\lambda}\bigg(\frac{u}{t}\bigg) - f_{\lambda}(x) \Bigg| dx.
\end{equation}
By an easy calculation,
$$f'_{\lambda}(x) = \frac{1}{x^2} \bigg( \frac{\lambda}{x} - 1\bigg) e^{-\lambda/x} = \frac{1}{\lambda^2} g\bigg(\frac{\lambda}{x}\bigg),$$
with $g(y) = y^2 (y-1) e^{-y}$, that is bounded on $[0,+\infty)$. Hence for some $M>0$,
$$f'_{\lambda}(x) \leq \frac{M}{\lambda^2} = O(t^{1/4}),$$
so that for $x \in [(u-1)/t,u/t]$,
\begin{equation}
\Bigg| f_{\lambda}\bigg(\frac{u}{t}\bigg) - f_{\lambda}(x) \Bigg| \leq C' t^{1/4} \bigg| \frac{u}{t} - x\bigg| \leq \frac{C'}{t^{3/4}},
\end{equation}
and finally
$$\sum_{u=t^{3/4}+1}^t \int_{(u-1)/t}^{u/t} \Bigg| f_{\lambda}\bigg(\frac{u}{t}\bigg) - f_{\lambda}(x) \Bigg| dx = O \bigg( \frac{1}{t^{3/4}} \bigg).$$
Putting everything together, we get
\begin{equation}
\rho_t(z) = \int_{t^{-1/4}}^1 \frac{1}{x} e^{-\|z\|^2/tx} dx + O \bigg( \frac{1}{t^{9/16}} \bigg).
\end{equation}
Making the change of variable $u = \|z\|^2/tx$, we obtain
\begin{equation}
\rho_t(z) = \int_{\|z\|^2/t}^{\|z\|^2/t^{3/4}} \frac{e^{-u}}{u} du + O \bigg( \frac{1}{t^{9/16}} \bigg),
\end{equation}
which allows us to conclude, since
\begin{equation}
\int_{\|z\|^2/t^{3/4}}^{+\infty} \frac{e^{-u}}{u} du = O \bigg( \frac{t^{3/4}}{\|z\|^2} e^{- \|z\|^2/t^{3/4}} \bigg) = O (t^{-1/8} e^{- t^{1/8}}).
\end{equation}
\end{proof}

Since $\bar{q}_{\mu,t}(r)$ is strictly decreasing, tends to $0$ as $r \to \infty$, and to $1$ as $r \to 0^+$, we can introduce
\begin{equation}
r_{\mu,t}^* = (\bar{q}_{\mu,t})^{-1}(1/2),
\end{equation}
and an easy calculation shows in this case that
\begin{equation}
r_{\mu,t}^* =  \sqrt{F^{-1}\bigg( \frac{2 \pi \log 2}{\mu \sqrt{3}} \bigg)} \sqrt{t},
\end{equation}
with $F(x) = \int_x^{+\infty} \frac{e^{-u}}{u} du$, that is a decreasing bijection from $(0,+\infty)$ onto itself. We also consider
\begin{equation}
r_{\mu,t}^- = (\bar{q}_{\mu,t})^{-1}(3/4) \quad \text{ and } \quad r_{\mu,t}^+ = (\bar{q}_{\mu,t})^{-1}(1/4),
\end{equation}
and we have
\begin{equation}
r_{\mu,t}^{\pm} = C^{\pm} \sqrt{t}.
\end{equation}
If we compute the derivative of $\bar{q}_{\mu,t}(r)$ with respect to $r$,
\begin{align*}
\frac{\partial}{\partial r} \bar{q}_{\mu,t}(r) = (1 - \bar{q}_{\mu,t}(r)) \times \frac{\sqrt{3} \mu}{2 \pi} \times \frac{e^{-r^2/t}}{r^2/t} \times \frac{2r}{t},
\end{align*}
we get, similarly to Eq.(\ref{ineq_p}), that
\begin{equation}
\bar{q}_{\mu,t}(r) = \frac{1}{2} - \frac{\epsilon(r)}{\sqrt{t}} \big(r - r_{n,t}^*\big),
\end{equation}
for $r \in [r_{\mu,t}^-,r_{\mu,t}^+]$, where $C_1 \leq \epsilon(r) \leq C_2$. We are thus in a position to apply the same construction as in the previous section (Figure \ref{gradient}), and the properties of Theorem \ref{main_theorem} (case (\ref{dense})) still hold: for any small $\epsilon>0$, with probability tending to $1$ as $t \to \infty$,
\begin{enumerate}[(a)]
\item there is a unique interface $\gamma_t$ surrounding $\partial S_{t^{\epsilon}}$,

\item this interface remains localized in the annulus $S_{r_{\mu,t}^* - t^{2/7+\epsilon},r_{\mu,t}^* + t^{2/7+\epsilon}}$ around $\partial S_{r_{\mu,t}^*}$, and its fluctuations compared to the circle $r=r_{\mu,t}^*$ are larger than $t^{2/7-\epsilon}$, both inward and outward,

\item its discrete length $L_t$ satisfies $t^{5/7-\epsilon} \leq L_t \leq t^{5/7+\epsilon}$.
\end{enumerate}

\begin{remark}
We could also modify our model by adding a fixed number $n_0$ of particles at each time $t$: the same results apply in this case, by using a Poissonian approximation in the same set $A$ of sites as in the previous section. Indeed, we can ensure that the configuration in $A$ coincides with the ``Poissonian configuration'' with probability at least $1 - \rho_t(A)$, for
\begin{equation}
\rho_t(A) = \sum_{z \in A} \rho_t(z) = \sum_{u=0}^t \pi_u(A).
\end{equation}
\end{remark}

\section*{Acknowledgements}

The author would like to thank C. Newman, B. Sapoval, and W. Werner for many stimulating discussions. This research was supported in part by the NSF under grant OISE-07-30136.

\end{document}